\documentclass[11pt]{article}
\usepackage{graphicx}
\usepackage{mathrsfs}

\usepackage{amsfonts}
\usepackage{amsmath}
\usepackage{amssymb}
\newtheorem{thm}{Theorem}[section]
\newtheorem{lem}[thm]{Lemma}

\numberwithin{equation}{section}

\newcommand{\ba}{\begin{array}}
\newcommand{\ea}{\end{array}}
\newcommand{\bt}{\begin{tabular}}
\newcommand{\et}{\end{tabular}}
\newcommand{\btb}{\begin{table}}
\newcommand{\etb}{\end{table}}
\newcommand{\bc}{\begin{center}}
\newcommand{\ec}{\end{center}}
\newcommand{\bea}{\begin{eqnarray}}
\newcommand{\eea}{\end{eqnarray}}
\newcommand{\Bea}{\begin{eqnarray*}}
\newcommand{\Eea}{\end{eqnarray*}}
\newcommand{\beq}{\begin{equation}}
\newcommand{\eeq}{\end{equation}}

\begin{document}
\baselineskip 16.5pt

\title{The 2-adic valuations of differences of Stirling numbers of the second kind}
\author{Wei Zhao\\
{\it Mathematical College, Sichuan University, Chengdu 610064, P.R. China}\\
Email: zhaowei9801@163.com\\
\\
Jianrong Zhao\\
{\it School of Economic Mathematics, Southwestern University of }\\
{\it Finance and Economics, Chengdu 610074, P.R. China}\\
Email: mathzjr@swufe.edu.cn, mathzjr@foxmail.com\\
\\
Shaofang Hong
\thanks{S. Hong is the corresponding author and was supported
partially by National Science Foundation of China Grant \#11371260.}
\\
{\it Mathematical College, Sichuan University, Chengdu 610064, P.R. China}\\
Email: sfhong@scu.edu.cn, s-f.hong@tom.com, hongsf02@yahoo.com\\
}
\date{}

\maketitle

\noindent{\bf Abstract.} Let $m, n, k$ and $c$ be positive integers.
Let $\nu_2(k)$ be the 2-adic valuation of $k$. By $S(n,k)$ we denote
the Stirling numbers of the second kind. In this paper, we first
establish a convolution identity of the Stirling numbers of the second
kind and provide a detailed 2-adic analysis to the Stirling numbers
of the second kind. Consequently, we show that if $2\le m\le n$ and
$c$ is odd, then $\nu_2(S(c2^{n+1},2^m-1)-S(c2^n, 2^m-1))=n+1$
except when $n=m=2$ and $c=1$, in which case $\nu_2(S(8,3)-S(4,3))=6$.
This solves a conjecture of Lengyel proposed in 2009.

\vspace{1mm}

\vspace{1mm}

\noindent {\small \textbf{MR(2000) Subject Classification:} Primary
11B73, 11A07}


\section{Introduction and the statements of main results}
Let $\mathbb{N}$ denote the set of nonegative integers and
let $n, k\in\mathbb{N}$. The Stirling numbers of the second
kind $S(n,k)$ is defined  as the number of ways to partition
a set of $n$ elements into exactly $k$ non-empty subsets.
Divisibility properties of Stirling numbers of the
second kind  $S(n,k)$ have been studied from a number of different
perspectives. For each given $k$, the sequence $\{S(n,k),n\ge k\}$
is known to be periodic modulo prime powers. Carlitz \cite{[Ca]}
and Kwong \cite{[Kw]} have studied the length of this period,
respectively.  Chan and Manna \cite{[CM]} characterized $S(n,k)$
modulo prime powers in terms of binomial coefficients when $k$ is a
multiple of prime powers. Various congruences involving sums of
$S(n,k)$ are also known \cite{[S]}.

Given a prime $p$ and a positive integer $m$, there exist unique
integers $a$ and $n$, with $p\nmid a$ and $n\geq 0$, such that
$m=ap^n$. The number $n$ is called the {\it $p$-adic valuation}
of $m$, denoted by $n=v_p(m)$. The study of $p$-adic valuations
of Stirling numbers of the second kind is full with
challenging problems. The values $\min\{v_p(k!S(n,k)): m\leq k\leq n\}$
are important in algebraic topology, see, for example,
\cite{[BD],[CK],[D2], [D3], [D4], [Lu1], [Lu2]}. Some work
evaluating $v_p(k!S(n,k))$ have appeared in above papers as well as
in \cite{[Cl],[D1],[Y]}. Lengyel  \cite{[Le1]} studied the 2-adic
valuations of $S(n,k)$ and conjectured, proved by Wannemacker
\cite{[W]}, that $\nu_2(S(2^n,k))=s_2(k)-1,$ where $s_2(k)$ means the
base $2$ digital sum of $k$. Lengyel \cite{[Le2]} showed that if
$1\leq k\leq 2^n$, then $\nu_2(S(c2^n,k))=s_2(k)-1$ for any positive
integer $c$. Hong et al \cite{[HZZ]} proved that
$\nu_2(S(2^n+1,k+1))=s_2(k)-1$, which confirmed a conjecture of
Amdeberhan et al \cite{[AMM]}.

On the other hand, Lengyel \cite{[Le2]} studied the 2-adic
valuations of the difference $S(c2^{n+1},k)-S(c2^{n+1},k)$. It
appears that its 2-adic valuation increases by one as $n$ by one,
provided that $n$ is
large enough. As a consequence, Lengyel proposed the following conjecture.\\

\noindent{\bf Conjecture 1.1.} {\it\cite{[Le2]} Let $n, k, a, b, c\in
\mathbb{N}$ with $c\geq 1$ being odd and $3\leq k\leq 2^n$. Then
\begin{align}\label{ad1}
\nu_2(S(c2^{n+1},k)-S(c2^{n},k))=n+1-f(k)
\end{align}
and
\begin{align}\label{ad2}
\nu_2(S(a2^{n},k)-S(b2^{n},k))=n+1+\nu_2(a-b)-f(k)
\end{align}
for some function $f(k)$ which is independent of $n$ (for any
sufficiently large $n$).}\\
\\
When $k$ is a power of 2 minus 1, Lengyel suggested
the following conjecture which is stronger than (\ref{ad1}).\\

\noindent{\bf Conjecture 1.2.} {\it\cite{[Le2]}
Let $c, m, n\in \mathbb{N}$ with $c\geq
1$ being odd and $2\le m\le n$. Then}
\begin{align*}
\nu_2(S(c2^{n+1},2^m-1)-S(c2^n, 2^m-1))=n+1.
\end{align*}

Lengyel \cite{[Le2]} showed that (\ref{ad1}) is true if $s_2(k)\le
2$. For any real number $x$, as usual, let $\lceil x\rceil$ and
$\lfloor x\rfloor$ denote the smallest integer no less than $x$ and
the biggest integer no more than $x$, respectively. In \cite{[ZHZ]},
we used the Junod's congruence \cite{[J]} about the Bell polynomials
to show the following result.
\begin{thm}\label{lem10}
\cite{[ZHZ]} Let $n,k, a, b, c\in \mathbb{N}$ with $c\geq 1$
being odd, $3\leq k\leq 2^n$ and $a>b$. If $k$ is not a power
of 2 minus 1, then
\begin{align*}
v_2(S(a2^{n},k)-S(b2^{n},k))=n+v_2(a-b)-\lceil\log_2k\rceil
+s_2(k)+\delta(k),
\end{align*}
where $\delta(4)=2$, $\delta(k)=1$ if $k>4$ is a power of 2, and
$\delta(k)=0$ otherwise. In particular,
\begin{align*}
v_2(S(c2^{n+1},k)-S(c2^{n},k))=n-\lceil\log_2k\rceil
+s_2(k)+\delta(k).
\end{align*}
\end{thm}
Therefore Conjecture 1.1 is true except when $k$ is a power of
$2$ minus 1, in which case Conjecture 1.1 is still kept open so far.
It is also remarked in \cite{[ZHZ]} that the techniques there is not
suitable for the remaining case that $k$ is a power of $2$ minus 1.

In this paper, we introduce a new method to investigate the 2-adic
valuations of differences of Stirling numbers of the second kind.
Our main goal in this paper is to study Conjecture 1.1 for the
remaining case and Conjecture 1.2.
We will develop a detailed 2-adic analysis to the Stirling numbers of
the second kind. The main results of this paper can be stated as follows.
\begin{thm}\label{thm0}
Let $a, b, n\in \mathbb{N}$ with $a>b\ge1$. Then each of the
following is true.

{\rm (i).} If $n\ge2$, then
\begin{align*}
\nu_2(S(a2^{n},3)-S(b2^n, 3)) \left\{\begin{array}{cl}
=n+1+\nu_2(a-b)& {\it if}~b2^n>n+2+\nu_2(a-b),\\
>n+1+\nu_2(a-b)& {\it if}~b2^n=n+2+\nu_2(a-b),\\
=b2^n-1& {\it if}~b2^n<n+2+\nu_2(a-b).
\end{array}
\right.
\end{align*}

{\rm (ii).} If $n\ge3$, then
\begin{align*}
\nu_2(S(a2^{n},7)-S(b2^n, 7)) \left\{\begin{array}{cl}
=n+1+\nu_2(a-b)& {\it if}~b2^n>n+3+\nu_2(a-b),\\
>n+1+\nu_2(a-b)& {\it if}~b2^n=n+3+\nu_2(a-b),\\
=b2^n-2& {\it if}~b2^n<n+3+\nu_2(a-b).
\end{array}
\right.
\end{align*}
\end{thm}

\begin{thm}\label{thm1} Let $c, m, n\in \mathbb{N}$ with $c\geq
1$ being odd and $2\le m\le n$. Then
\begin{align*}
\nu_2(S(c2^{n+1},2^m-1)-S(c2^n, 2^m-1))=n+1
\end{align*}
except when $n=m=2$ and $c=1$, in which case one has $\nu_2(S(8,3)-S(4,3))=6$.
\end{thm}
Evidently, by Theorem \ref{thm1} one knows that Conjecture 1.2 is true except
for the exceptional case that $n=m=2$ and $c=1$, in which case Conjecture 1.2
is not true. It also implies that (\ref{ad1}) holds for the remaining
case that $k$ equals a power of 2 minus 1. By Theorem \ref{thm0},
we know that (\ref{ad2}) is true for the cases that $k=3$ and 7
and sufficiently large $n$, but the truth of (\ref{ad2}) still
keeps open when $k$ is a power of 2 minus 1 and no less than 15.

This paper is organized as follows. In Section 2, we recall some
known results and show also several new results that are needed
in the proof of Theorems  \ref{thm0} and \ref{thm1}. The proofs
of Theorems \ref{thm0} and \ref{thm1} are given in Section 3.
The key new ingredients in this paper are to make use of a
classical congruence about binomial coefficients and to
establish a new convolution identity of
the Stirling numbers of the second kind.

In ending this section, we list several elementary properties of
$S(n, k)$ that will be used freely throughout this paper:

$\bullet$ The recurrence relation
\begin{align}\label{eq1}
S(n,k)=S(n-1,k-1)+kS(n-1,k),
\end{align}
with initial condition $S(0,0)=1$ and $S(n,0)=0$ for $n>0$.

$\bullet$ The explicit formula
\begin{align}\label{1.1}
S(n,k)=\frac{1}{k!}\sum_{i=0}^k(-1)^i{k\choose i}(k-i)^n.
\end{align}

$\bullet$ The generating function
\begin{align}\label{q1}
(e^t-1)^k=k!\sum_{j=k}^{\infty}S(j,k)\frac{t^j}{j!}
\end{align}
and
\begin{align}\label{05}
\frac{1}{(1-x)(1-2x)....(1-kx)}=\sum_{j=0}^{\infty}S(j+k,k)x^j.
\end{align}

\section{Preliminary lemmas}

In the present section, we give preliminary lemmas which
are needed in the proofs of Theorems \ref{thm0} and \ref{thm1}.
We begin with Kummer's identity.

\begin{lem}\label{lem01}\cite{[Ku]} (Kummer)
Let $k$ and $n\in\mathbb{N}$ be such that $k\le n$. Then
$$\nu_2\Big({n\choose k}\Big)=s_2(k)+s_2(n-k)-s_2(n).$$ Moreover,
$s_2(k)+s_2(n-k)\ge s_2(n).$
\end{lem}

The following classical congruence about binomial coefficients
is given in \cite{[R]}, which is the first new
key ingredient in the proof of Theorem 1.3.

\begin{lem}\cite{[R]}\label{lem02}
Let $n$ and $k$ be positive integers. For all primes $p$, we have
$${pn\choose pk}\equiv{n\choose k}\mod{pn\mathbb{Z}_p},$$
where $\mathbb{Z}_p$ stands for the ring of $p$-adic integers.
\end{lem}

\begin{lem}\label{lem03} \cite{[HZZ]}
Let $N\geq 2$ be an integer and $r$, $t$ be odd numbers. For any
$m\in \mathbb{Z}^+,$  we have $\nu_2((r2^N-1)^{t2^m}-1)=m+N.$
\end{lem}

\begin{lem}\cite{[CM]}\label{lem04}
Let $a,n$ and $m\ge3$  be positive integers with $n>a2^m$. Then
\begin{align*}
S(n,a2^m)\equiv \left\{
\begin{array}{lc}
a2^{m-1}{\frac{n-1}{2}-a2^{m-2}-1\choose a2^{m-2}-1}, &2\nmid n  \\
\\
a2^{m-1}{\frac{n}{2}-a2^{m-2}-2\choose
a2^{m-2}-1}+{\frac{n}{2}-a2^{m-2}-1\choose a2^{m-2}-1}, &2\mid n
\end{array}
\right.\mod {2^m}.
\end{align*}
\end{lem}

\begin{lem}\label{lem05}
Let $a,n$ and $m\ge3$  be positive integers with $n\ge a2^m$. Then
\begin{align*}
S(n,a2^m-1)\equiv \left\{
\begin{array}{lc}
a2^{m-1}{\frac{n}{2}-a2^{m-2}-1\choose a2^{m-2}-1}, &2\mid n  \\
\\
a2^{m-1}{\frac{n+1}{2}-a2^{m-2}-2\choose
a2^{m-2}-1}+{\frac{n+1}{2}-a2^{m-2}-1\choose a2^{m-2}-1}, &2\nmid n
\end{array}
\right.\mod {2^m}.
\end{align*}
\end{lem}

\begin{proof}
Using the recurrence relation (\ref{eq1}), we know that
\begin{align}\label{q5}
S(n+1,a2^m)=S(n,a2^m-1)+a2^mS(n,a2^m).
\end{align}
Thus  Lemma \ref{lem05} follows immediately from (\ref{q5}) and
Lemma \ref{lem04}.
 \hfill$\Box$
\end{proof}

\begin{lem}\cite{[Le2]}\label{lem06} 
 Let $m,n,c\in\mathbb{N}$ and  $0\le m<n$.  Then
$\nu_2(S(c2^n+2^m,2^{n}))=n-1-m.$
\end{lem}

\begin{lem}\label{lem07}\cite{[Le2]}
Let $c,n,m\in\mathbb{N}$. If $2\le m\le n$ and $c\ge 1$, then
 $S(c2^n,2^m)\equiv 1\mod 4$ and $S(c2^n,2^m-1)\equiv 3\cdot 2^{m-1}\mod {2^{m+1}}.$
\end{lem}

\begin{lem}\cite{[Le2]}\label{lem08}
Let $m$ be a positive integer. Then for $m\ge3$, one has
$$\prod_{i=1}^{2^{m-1}}(1-(2i-1)x)\equiv(1+3x^2)^{2^{m-2}}\mod {2^{m+1}},$$
and for $m\ge4$, one has
$$\prod_{i=1}^{2^{m-1}-1}(1-2ix)\equiv1+2^{m-1}x+2^{m-1}x^2+2^mx^4\mod {2^{m+1}}.$$
\end{lem}

\begin{lem}\label{lem09}
Let $c,n$ and $m$ be positive integers with $3\le m\le n$. Then
\begin{align}\label{eq01}
S(c2^n+2^{m-1},2^m)\equiv 3\mod 4
\end{align}
and
\begin{align}\label{eq02} S(c2^n+2^{m-1},2^m-1)\equiv 2^{m-1}\mod
{2^{m+1}}.\end{align}
\end{lem}

\begin{proof} For any integer $n\ge m\ge3$, we deduce that
\begin{align}\label{eq03}
{c2^{n-1}-1\choose2^{m-2}-1}&=\frac{(c2^{n-1}-1)(c2^{n-1}-2)
...(c2^{n-1}-2^{m-2}+1)}{(2^{m-2}-1)!}\nonumber\\
&=-1+\sum_{t=1}^{2^{m-2}-1}(c2^{n-1})^t(-1)^{2^{m-2}-1-t} \sum_{1\le
i_1<...<i_t\le 2^{m-2}-1}\frac{1}{i_1...i_t}\nonumber\\
&=-1+\sum_{t=1}^{2^{m-2}-1}c^t(-1)^{t+1} \sum_{1\le i_1<...<i_t\le
2^{m-2}-1}\frac{2^{t(n-1)}}{i_1...i_t}.
\end{align}
Obviously, $\nu_2(i)\le m-3$ for any integer $1\le i\le  2^{m-2}-1$.
So for any integers $i_1, ..., i_t$ with $1\le i_1<...<i_t\le 2^{m-2}-1$,
we have
\begin{align}\label{eq0301}
\nu_2\Big(\frac{2^{t(n-1)}}{i_1...i_t}\Big)\ge t(n-1)-t(m-3)\ge
t(n-m+2)\ge 2t\ge2
\end{align}
since $n\ge m\ge 3$. It then follows from  Lemma \ref{lem04},
(\ref{eq03}) and (\ref{eq0301}) that
\begin{align*}
S(c2^n+2^{m-1},2^m)\equiv{c2^{n-1}-1\choose 2^{m-2}-1}\equiv3 \mod
4,
\end{align*}
which means that (\ref{eq01}) is true.

Now we prove that congruence (\ref{eq02}) is true. For $m=3$, one
has
\begin{align}\label{eq0302}
\prod_{i=1}^{2^{m-1}-1}(1-2ix)=(1-2x)(1-4x)(1-6x)\equiv1+4x+12x^2\mod
{2^{4}}.
\end{align}
So by Lemma \ref{lem08} and (\ref{eq0302}), one get
\begin{align}\label{eq0303}
\prod_{i=1}^{2^{m-1}-1}(1-2ix)\equiv1+2^{m-1}x+2^{m-1}x^2
\phi_m(x)\mod {2^{m+1}},
\end{align}
where  $\phi_1(x):=3$ and $\phi_m(x):=1+2x^2$ if $m\ge4$. Then by
(\ref{05}), (\ref{eq0303}) and Lemma \ref{lem08}, we get the
following congruence modulo $2^{m+1}$:
\begin{align}\label{06}
\sum_{j=0}^{\infty}S(j+2^m-1,2^{m}-1)x^j
=&\prod_{i=1}^{2^m-1}\frac{1}{1-ix}\nonumber\\
=&\Big(\prod_{i=1}^{2^{m-1}-1}\frac{1}{1-2ix}\Big)
\Big(\prod_{i=1}^{2^{m-1}}\frac{1}{1-(2i-1)x}\Big)\nonumber\\
\equiv& (1+2^{m-1}x+2^{m-1}x^2\phi_m(x))^{-1}(1+3x^2)^{-2^{m-2}}\nonumber\\
\equiv &(1+3\cdot2^{m-1}x+3\cdot2^{m-1}x^2\phi_m(x))
\sum_{i=0}^\infty{-2^{m-2}\choose i}3^ix^{2i}.
\end{align}
Note that the coefficient of $x^{c2^n-2^{m-1}+1}$ is
$S(c2^n+2^{m-1},2^m-1)$.  By (\ref{06}), we have
\begin{align}\label{eq04}
S(c2^n+2^{m-1},2^m-1)&\equiv3\cdot 2^{m-1}\cdot
3^{c2^{n-1}-2^{m-2}}{-2^{m-2}\choose c2^{n-1}-2^{m-2}}\mod
{2^{m+1}}.
\end{align}
Then Lemma \ref{lem03} tells us that
\begin{align}\label{eq05}
3^{c2^{n-1}-2^{m-2}}=(2^2-1)^{(c2^{n-m+1}-1)2^{m-2}}
\equiv1\mod{2^m}.
\end{align}
By (\ref{eq03}), we obtain that
\begin{align}\label{eq06}
{-2^{m-2}\choose
c2^{n-1}-2^{m-2}}=(-1)^{c2^{n-1}-2^{m-2}}{c2^{n-1}-1\choose
c2^{n-1}-2^{m-2}}={c2^{n-1}-1\choose 2^{m-2}-1}\equiv3\mod{4}.
\end{align}
It then follows from (\ref{eq04}) to (\ref{eq06}) that (\ref{eq02}) holds.
The proof of Lemma \ref{lem09} is complete. \hfill $\Box$
\end{proof}\\

Now we give a new convolution identity about Stirling numbers of the
second kind. It is the second new key ingredient in the proof of
Theorem 1.3.

\begin{lem}\label{lem11} Let $k_1$, $k_2$ and $n$ be positive integers. Then
\begin{align*}
S(n,k_1+k_2) = \frac{k_1!k_2!}{(k_1+k_2)!}\sum_{i=k_1}^{n-k_2}
{n\choose i}S(i,k_1)S(n-i,k_2).
\end{align*}
\end{lem}
\begin{proof}
By (\ref{q1}),  we have
\begin{align}\label{q01}
(e^t-1)^{k_1+k_2}&=(e^t-1)^{k_1}\cdot(e^t-1)^{k_2}\nonumber \\
&=\Big(k_1!\sum_{n=k_1}^{\infty}S(n,k_1)\frac{t^n}{n!}\Big)
\cdot\Big(k_2!\sum_{n=k_2}^{\infty}S(n,k_2)\frac{t^n}{n!}\Big)\nonumber \\
&=k_1!k_2!\Big(\sum_{n=k_1+k_2}^{\infty} \Big(\sum_{i=k_1}^{n-k_2}
\frac{S(i,k_1)S(n-i,k_2)}{i!(n-i)!}\Big)t^n\Big).
\end{align}

On the other hand, by (\ref{q1}) we know that
\begin{align}\label{q02}
(e^t-1)^{k_1+k_2}=(k_1+k_2)!\sum_{n=k_1}^{\infty}S(n,k_1+k_2)\frac{t^n}{n!}.
\end{align}
Then comparing the coefficients of $t^n$ in (\ref{q01}) and (\ref{q02}),
we obtain that
\begin{align}\label{q03}
S(n,k_1+k_2)\frac{(k_1+k_2)!}{n!} =k_1!k_2! \sum_{i=k_1}^{n-k_2}
\frac{S(i,k_1)S(n-i,k_2)}{i!(n-i)!}.
\end{align}
It then follows from (\ref{q03}) that
\begin{align*}
S(n,k_1+k_2) = \frac{k_1!k_2!}{(k_1+k_2)!}\sum_{i=k_1}^{n-k_2}
{n\choose i}S(i,k_1)S(n-i,k_2)
\end{align*}
as desired. The proof of Lemma \ref{lem11} is complete. \hfill$\Box$
\end{proof}

\begin{lem}\label{lem12}
Let $c,i,s,t\in\mathbb{N}$ with $c\ge 1$ being odd. Then each of the
following is true.

{\rm (i).} If $1\le i<c2^s$ and $\nu_2(i)<s$, then
$$\nu_2\Big({c2^s\choose i}\Big)\ge s-\nu_2(i),$$
with equality holding if and only if $s_2(c-1-\lfloor
i/2^s\rfloor)+s_2(\lfloor i/2^s\rfloor)=s_2(c-1).$

{\rm (ii).} If $3\le i<c2^s$ and $t\ge 2$, then
$$\nu_2\Big(2^{ti}{c2^s\choose i}\Big)\ge s+6.$$
\end{lem}
\begin{proof}
(i). For any integer $i$ with $i<c2^s$ and $\nu_2(i)<s$, we can
write $i=a+b2^s$ with $a,b\in\mathbb{N}$ and $0<a<2^s$. Then one has
$b=\lfloor \frac{i}{2^s}\rfloor$ and $\nu_2(i)=\nu_2(a)$. Using
Lemma \ref{lem01}, we derive that
\begin{align}\label{q6}
\nu_2\Big({c2^s\choose i}\Big)&=s_2(c2^s-i)+s_2(i)-s_2(c2^s)\nonumber\\
&=s_2(c2^s-a-b2^s)+s_2(a+b2^s)-s_2(c2^s)\nonumber\\
&=s_2((c-1-b)2^s+2^s-a)+s_2(a)+s_2(b)-s_2(c)\nonumber\\
&=s_2(c-1-b)+s_2(b)+s_2(2^s-a)+s_2(a)-s_2(c).
\end{align}
Note that $s_2(2^s-a)=s-\nu_2(a)-s_2(a)+1$. Since $c\ge 1$ is odd, one has
$s_2(c)-s_2(c-1)=1$. It then follows from Lemma \ref{lem01} and (\ref{q6}) that
\begin{align*}
\nu_2\Big({c2^s\choose i}\Big)\ge
s_2(c-1)+s_2(2^s-a)+s_2(a)-s_2(c)=s-\nu_2(a)=s-\nu_2(i),
\end{align*}
and equality holds if and only if $s_2(c-1-b)+s_2(b)=s_2(c-1).$ This
implies that part (i) is true.  So part (i) is proved.

(ii). Clearly, if $\nu_2(i)\ge s$, then $\nu_2\big({c2^s\choose
i}\big)\ge 0\ge s-\nu_2(i).$ If $\nu_2(i)<s$, by part (i) one also
has $\nu_2\big({c2^s\choose i}\big)\ge s-\nu_2(i).$ Then we deduce
that
\begin{align}\label{q601}
\nu_2\Big(2^{ti}{c2^s\choose i}\Big)&\ge ti+s-\nu_2(i).
\end{align}
Since $t\ge 2$, by (\ref{q601}), one can easily check that
$\nu_2\big(2^{ti}{c2^s\choose i}\big)\ge s+6$ if $i\le 5.$ This
means that part (ii) is true for the case $i\le5$. Now let $i\ge6$. By
(\ref{q601}), we get that
\begin{align*}
\nu_2\Big(2^{ti}{c2^s\choose i}\Big)\ge (t-1)i+s+i-\nu_2(i)\ge s+6
\end{align*}
as desired.  \hfill$\Box$
\end{proof}

\begin{lem}\label{lem13}
Let $c,r$ and $s$ be positive integers with $c\ge 1$ being odd and
$s\ge r\ge 3$. For any integer $i$ with $2^r-1\le i\le c2^s-2^r$, define
\begin{align}\label{ad3}
f_{r,s}(i):={c2^s\choose i}S(i,2^r-1)S(c2^s-i,2^r).
\end{align}
Then each of the following is true.

{\rm (i).} If $\nu_2(i)\le r-2$, then $\nu_2(f_{r,s}(i))\ge s+2$.

{\rm (ii).} If $\nu_2(i)=r-1$, then
$\nu_2(f_{r,s}(i))\ge s$ with equality holding if and only if
$$s_2(c-1-\lfloor i/2^s\rfloor)+s_2(\lfloor i/2^s\rfloor)=s_2(c-1).$$

{\rm (iii).}
$\nu_2\Bigg(\displaystyle\sum_{i=2^r-1
\atop \nu_2(i)\le r-1}^{c2^s-2^r}f_{r,s}(i)\Bigg)\ge s+1$.

{\rm (iv).} $f_{r,s}(2^r)+f_{r,s}(c2^s-2^r)\equiv 2^s\mod{2^{s+1}}$.

{\rm (v).} For any integer $l$ with $1\le l\le c2^{s-r}-2$, we have
$$f_{r,s+1}(l2^{r+1}+2^r)\equiv f_{r,s}(l2^{r}+2^{r-1})\mod{2^{s+2}}.$$
\end{lem}

\begin{proof}
(i). If $\nu_2(i)\le r-3$, then using Lemmas \ref{lem04},
\ref{lem05} and \ref{lem12} (i), we deduce that
\begin{align}\label{02}
\nu_2(f_{r,s}(i))&=\nu_2\Big({c2^s\choose i}\Big)
+\nu_2\big(S(i,2^r-1)S(c2^s-i,2^r)\big)\nonumber\\
&\ge s-\nu_2(i)+r-1\ge s+2.
\end{align}
So part (i) is true if $\nu_2(i)\le r-3$.

Now we let $\nu_2(i)=r-2$. Since $2^r-1\le i$, one may let
$i=2^{r-2}+i_12^{r-1}$ with $i_1\ge2$. Thus by Lemma \ref{lem01} we obtain that
\begin{align}\label{q10}
&\nu_2\Big({\frac{i}{2}-2^{r-2}-1\choose 2^{r-2}-1}\Big)
=\nu_2\Big({i_12^{r-2}-2^{r-3}-1\choose 2^{r-2}-1}\Big)\nonumber\\
&=s_2(2^{r-2}-1)+s_2((i_1-1)2^{r-2}-2^{r-3})-s_2(i_12^{r-2}-2^{r-3}-1)\nonumber\\
&=(r-2)+s_2((i_1-2)2^{r-2}+2^{r-3})-s_2((i_1-1)2^{r-2}+2^{r-3}-1)\nonumber\\
&=2+s_2(i_1-2)-s_2(i_1-1)\nonumber\\&=1+s_2(1)+s_2(i_1-2)-s_2(i_1-1)\ge 1.
\end{align}
Clearly $i$ is even since $\nu_2(i)=r-2$ and $r\ge 3$. From Lemma
\ref{lem05}, we get that
\begin{align}\label{q101}
S(i,2^r-1)\equiv 2^{r-1}{\frac{i}{2}-2^{r-2}-1\choose 2^{r-2}-1}\mod
{2^r}.
\end{align}
It then follows from (\ref{q10}) and (\ref{q101}) that $\nu _2(S(i,2^r-1))\ge r$.
Since $\nu_2(i)=r-2$, by Lemma \ref{lem12} one has
\begin{align}\label{q11}
\nu_2(f_{r,s}(i))&=\nu_2\Big({c2^s\choose i}\Big)
+\nu_2\big(S(i,2^r-1)S(c2^s-i,2^r)\big)\nonumber\\
&\ge s-\nu_2(i)+r= s-(r-2)+r\nonumber\\
&\ge s+2.
\end{align}
By (\ref{q11}), part (i) holds
for the case $\nu_2(i)=r-2$. So part (i) is proved.

(ii). Let $\nu_2(i)=r-1$. Then we can write $i=2^{r-1}+i_22^r$ with
$i_2\ge1$ since $i\ge 2^r-1$. So Lemma \ref{lem05} tells us that
\begin{align}\label{q7}
S(i,2^r-1)=S(2^{r-1}+i_22^r, 2^r-1)\equiv2^{r-1}{i_22^{r-1}-1\choose
2^{r-2}-1} \mod{2^r}
\end{align}
since $r\ge 3$. Applying Lemma \ref{lem01}, we deduce that
\begin{align}\label{q701}
\nu_2\Big({i_22^{r-1}-1\choose 2^{r-2}-1}\Big)&=s_2(2^{r-2}-1)
+s_2(i_22^{r-1}-2^{r-2})-s_2(i_22^{r-1}-1)\nonumber\\
&=r-2+s_2((i_2-1)2^{r-1}+2^{r-2})-s_2((i_2-1)2^{r-1}+2^{r-1}-1)
\nonumber\\&=0.
\end{align}
It then follows from (\ref{q7}) and (\ref{q701}) that
\begin{align}\label{q8}
\nu_2(S(i,2^r-1))=r-1.
\end{align}

On the other hand,  by Lemma \ref{lem06}
and noticing that $s\ge r$, we get that
\begin{align}\label{q9}
\nu_2(S(c2^s-i,2^r))=\nu_2(S(c2^s-i_22^r-2^{r-1}, 2^r))=0.
\end{align}
Then from Lemma \ref{lem12} (i),  (\ref{q8}) and (\ref{q9}), we deduce
that $\nu_2(f_{r,s}(i))\ge s$. Furthermore, the equality holds if
and only if $s_2(c-1-\lfloor i/2^s\rfloor)+s_2(\lfloor
i/2^s\rfloor)=s_2(c-1)$. Thus part (ii) is proved.

(iii). We define a subset of positive integers as follows:
$$J=\{i|2^r<i<c2^s-2^r, \nu_2(i)=r-1, s_2(c-1-\Big\lfloor
\frac{i}{2^s}\Big\rfloor)+s_2(\Big\lfloor
\frac{i}{2^s}\Big\rfloor)=s_2(c-1)\}.$$
We claim that $|J|$ is even.
Note that $J=\emptyset$ if $c=1$ and $r\le s \le r+1$. If $c=1$ and
$s\ge r+2$, then $J=\{i|2^r<i<2^s-2^r,\,\nu_2(i)=r-1\}.$ One can
easily compute that $|J|=2^{s-r}-2$. So the claim is true for the case
$c=1$. In what follows we deal with the case that $c\ge3$ is odd.

Let $c\ge3$ be an odd integer. One can define three subsets of $J$ as follows:
$$J_1=\{i\in J|2^r<i<2^s\},\,J_2=\{i\in J|2^s<i<(c-1)2^s\},\,
J_3=\{i\in J|(c-1)2^s<i<c2^s-2^r\}.$$
Clearly, $2^s\not\in J$ and $(c-1)2^s\not\in J$. Thus $J_1$, $J_2$ and
$J_3$ are disjoint and $J=J_1\cup J_2\cup J_3$, which implies that
$|J|=|J_1|+|J_2|+|J_3|$. If either $\lfloor \frac{i}{2^s}\rfloor=0$
or $\lfloor \frac{i}{2^s}\rfloor=c-1$, then
$s_2(c-1-\Big\lfloor \frac{i}{2^s}\Big\rfloor)+s_2(\Big\lfloor
\frac{i}{2^s}\Big\rfloor)=s_2(c-1)$. It follows that
$$J_1=\{i|2^r<i<2^s,\nu_2(i)=r-1\} \ {\rm and} \
J_3=\{i|(c-1)2^s<i<c2^s-2^r, \nu_2(i)=r-1\}.$$
One can compute that $|J_1|=|J_3|=2^{s-r}-1$. So to finish the proof of
the claim, it suffices to show that $|J_2|$ is even.
Now take any element $i\in J_2$. Then $2^s<i<(c-1)2^s$. Since $i=(i-\Big\lfloor
\frac{i}{2^s}\Big\rfloor2^s)+\Big\lfloor\frac{i}{2^s}\Big\rfloor2^s$,
one has $\nu_2(i)=\nu_2(i-\Big\lfloor \frac{i}{2^s}\Big\rfloor2^s)$
and $1\le \Big\lfloor \frac{i}{2^s}\Big\rfloor\le c-2$.
We can easily check that
$$i':=(i-\Big\lfloor\frac{i}{2^s}\Big\rfloor2^s)+(c-1-\Big\lfloor
\frac{i}{2^s}\Big\rfloor)2^s\in J_2.$$
Then $i\ne i'$. Otherwise, $i=i'$ implies that
$\lfloor \frac{i}{2^s}\rfloor=\frac{c-1}{2}$, and so
$$s_2(c-1-\Big\lfloor\frac{i}{2^s}\Big\rfloor)+s_2(\Big\lfloor
\frac{i}{2^s}\Big\rfloor)=2s_2\Big(\frac{c-1}{2}\Big)>s_2(c-1)$$
since $c\ge 3$ is odd. Hence $i\not\in J_2$. We arrive at a contradiction.
It then follows that $|J_2|$ is even. The claim is proved.

From part (i) and (ii) and the claim above, we derive that
\begin{align*}
 \sum_{i=2^r-1,\nu_2(i)\le r-1}^{c2^s-2^r}f_{r,s}(i)
 \equiv\sum_{i\in J}f_{r,s}(i)\equiv 0 \mod{2^{s+1}},
\end{align*}
which implies that part (iii) holds.

(iv). By Lemma \ref{lem01}, we get that
\begin{align}\label{q13}
 \nu_2\Big({c2^s\choose 2^r}\Big)&=\nu_2\Big({c2^s\choose c2^s-2^r}\Big)\nonumber\\
 &=s_2(c2^s-2^r)+s_2(2^r)-s_2(c2^s)\nonumber\\
 &=s_2((c-1)2^s)+s_2(2^s-2^r)+1-s_2(c)\nonumber\\
 &=s-r
\end{align}
since $s\ge r$ and $c\ge 1$ is odd. Then one may let
${c2^s\choose 2^r}=2^{s-r}(1+2k_0)$ with $k_0\in \mathbb{N}$.
From Lemma
\ref{lem07}, we deduce that there are three integers
$k_1$, $k_2$ and $k_3$ such that $S(2^r,2^r-1)=
2^{r-1}(3+4k_1)$, $S(c2^s-2^r,2^r-1)= 2^{r-1}(3+4k_2)$
and $S(c2^s-2^r,2^r)=1+4k_3.$ It then follows from   (\ref{q13})
that
\begin{align}\label{q14}
 f_{r,s}(2^r)&={c2^s\choose 2^r}S(2^r,2^r-1)S(c2^s-2^r,2^r)\nonumber\\
 &=2^{s-r}(1+2k_0)2^{r-1}(3+4k_1)(1+4k_3)\nonumber\\
 &\equiv 2^{s-1}(3+2k_0)\mod{2^{s+1}}
\end{align}
and
\begin{align}\label{q15}
 f_{r,s}(c2^s-2^r)&={c2^s\choose c2^s-2^r}S(c2^s-2^r,2^r-1)S(2^r,2^r)\nonumber\\
 &=2^{s-r}(1+2k_0)2^{r-1}(3+4k_2)\nonumber\\
 &\equiv 2^{s-1}(3+2k_0)\mod{2^{s+1}}.
\end{align}
Then the desired result follows immediately
from (\ref{q14}) and (\ref{q15}). Part (iv) is proved.

(v). Lemma \ref{lem02} gives us that
\begin{align}\label{07}
{c2^{s+1}\choose l2^{r+1}+2^r}\equiv {c2^{s}\choose
l2^{r}+2^{r-1}}\mod 2^{s+1}.
\end{align}
From Lemma \ref{lem12} (i), we obtain that
\begin{align}\label{08}
\nu_2\Big({c2^{s}\choose l2^{r}+2^{r-1}}\Big)\ge s-r+1.
\end{align}
It then follows from (\ref{07}), (\ref{08}) and Lemma \ref{lem07}
that
\begin{align}\label{09}
f_{r,s+1}(l2^{r+1}+2^r)&={c2^{s+1}\choose
l2^{r+1}+2^r}S(l2^{r+1}+2^r,2^r-1)S(c2^{s+1}-l2^{r+1}-2^r,2^r)\nonumber\\
&=\Big({c2^{s}\choose l2^{r}+2^{r-1}}+a_1\cdot
2^{s+1}\Big)2^{r-1}(3+4a_2)(1+4a_3)\nonumber\\
&\equiv 3\cdot 2^{r-1}{c2^{s}\choose l2^{r}+2^{r-1}}\mod 2^{s+2},
\end{align}
where $a_1,a_2,a_3\in \mathbb{N}$.

On the other hand, by  Lemma \ref{lem09} and (\ref{08}), we
derive that
\begin{align}\label{10}
f_{r,s}(l2^{r}+2^{r-1})&={c2^{s}\choose l2^{r}+2^{r-1}}
S(l2^{r}+2^{r-1},2^r-1)S(c2^{s}-l2^{r}-2^{r-1},2^r)\nonumber\\
&={c2^{s}\choose l2^{r}+2^{r-1}}2^{r-1}(1+4a_4)(3+4a_5) \nonumber\\
&\equiv3\cdot 2^{r-1}{c2^{s}\choose l2^{r}+2^{r-1}}\mod{2^{s+2}},
\end{align}
where $a_4,a_5\in \mathbb{N}$. So part (v) follows immediately from
(\ref{09}) and (\ref{10}).

 This completes the proof of Lemma \ref{lem13}. \hfill$\Box$
\end{proof}

\section{Proofs of Theorems 1.2 and 1.3}

In this section, we use the lemmas presented in previous section to
show Theorems \ref{thm0} and \ref{thm1}. We begin with the proof
of Theorem \ref{thm0}.\\

\noindent{\bf Proof of Theorem \ref{thm0}}. (i). By (\ref{1.1}), one
has
\begin{align} \label{a1}
S(a2^{n},3)-S(b2^{n},3)=\frac{1}{6}(3^{a2^n}-3^{b2^n})
-\frac{1}{2}(2^{a2^n}-2^{b2^n}).
\end{align}
Lemma \ref{lem03} tells us that
$\nu_2(3^{(a-b)2^n}-1)=n+\nu_2(a-b)+2$. Then
\begin{align} \label{a2}
\nu_2\big(\frac{1}{6}(3^{a2^n}-3^{b2^n})\big)=-1+
\nu_2\big(3^{b2^n-1}(3^{(a-b)2^n}-1)\big)=n+\nu_2(a-b)+1.
\end{align}
Since $a>b$, we get that
\begin{align} \label{a3}
\nu_2\big(\frac{1}{2}(2^{a2^n}-2^{b2^n})\big)=b2^n-1.
\end{align}
It then follows from (\ref{a1}) to (\ref{a3}) that part (i) is true.
Part (i) is proved.

(ii). Using (\ref{1.1}), we deduce that
\begin{align} \label{b1}
7!(S(a2^{n},7)-S(b2^{n},7))&=\sum_{i=0}^7(-1)^i{7\choose
i}((7-i)^{a2^n}-(7-i)^{b2^n})=h(1)+h(2),
\end{align}
where
\begin{align} \label{b2}
h(1):=7^{b2^n}(7^{(a-b)2^n}-1)+{7\choose 2}5^{b2^n}(5^{(a-b)2^n}-1)+
{7\choose 4}3^{b2^n}(3^{(a-b)2^n}-1)
\end{align}
and
\begin{align} \label{b3}
h(2):={7\choose 1}(6^{b2^n}-6^{a2^n})+{7\choose
3}(4^{b2^n}-4^{a2^n})+ {7\choose 5}(2^{b2^n}-2^{a2^n}).
\end{align}
Let $d:=n+\nu_2(a-b)$. Then there exists an odd integer $c$ such
that $(a-b)2^n=c2^d$. From Lemma \ref{lem12} (ii),  one knows that
$$\sum_{i=3}^{c2^d}(-1)^i8^i{c2^d\choose i}\equiv
\sum_{i=3}^{c2^d}(-1)^i4^i{c2^d\choose i}\equiv
\sum_{i=3}^{c2^d}4^i{c2^d\choose i}\equiv 0\mod 2^{d+6}.$$ It
follows that
\begin{align} \label{b4}
7^{c2^d}-1=(8-1)^{c2^d}-1=\sum_{i=1}^{c2^d}(-1)^i8^i{c2^d\choose
i}\equiv -c2^{d+3}-c2^{d+5}\mod{2^{d+6}},
\end{align}
\begin{align} \label{b5}
5^{c2^d}-1=(4+1)^{c2^d}-1=\sum_{i=1}^{c2^d}4^i{c2^d\choose i} \equiv
c2^{d+2}-c2^{d+3}\mod{2^{d+6}}
\end{align}
and
\begin{align} \label{b6}
3^{c2^d}-1=(4-1)^{c2^d}-1=\sum_{i=1}^{c2^d}(-1)^i4^i{c2^d\choose
i} \equiv -c2^{d+2}-c2^{d+3}\mod{2^{d+6}}.
\end{align}
Since $b\ge1$ and $n\ge3$, one has $7^{b2^n}\equiv 5^{b2^n} \equiv
3^{b2^n}\equiv 1 \mod 2^4$. Then by (\ref{b2}) and (\ref{b4}) to
(\ref{b6}), and noting that $c\ge1$ is odd, we get that
\begin{align}\label{b7}
h(1)&\equiv-c2^{d+3}-c2^{d+5}+{7\choose 2}
(c2^{d+2}-c2^{d+3})+{7\choose
4}(-c2^{d+2}-c2^{d+3})\nonumber\\
&\equiv c2^{d+5}\equiv 2^{d+5}\mod{2^{d+6}}.
\end{align}
Thus by (\ref{b7}), we have
\begin{align}\label{b8}
\nu_2(h(1))=d+5=n+5+\nu_2(a-b).
\end{align}

On the other hand,  since $a>b\ge1$ and $n\ge3$, one has
$a2^n>b2^n+3$ and $b2^{n+1}>b2^n+3$. So by (\ref{b3}), we obtain
that
\begin{align} \label{b9}
h(2)\equiv 7\cdot 6^{b2^n}+{7\choose 5}2^{b2^n}=7\cdot
2^{b2^n}(3^{b2^n}+3)\equiv 2^{b2^n+2}\mod 2^{b2^n+3}
\end{align}
since $3^{b2^n}\equiv1 \mod 2^{n+1}$ and $n\ge3$. Hence (\ref{b9})
tells us that
\begin{align} \label{b10}
\nu_2(h(2))=b2^n+2.
\end{align}
Note that $\nu_2(7!)=4$. It then follows from (\ref{b1}), (\ref{b8})
and (\ref{b10}) that part (ii) is true.

The proof of
Theorem \ref{thm0} is complete. \hfill $\Box$\\

In what follows, we give the proof of Theorem \ref{thm1}
as the conclusion of this paper.\\

\noindent{\bf Proof of Theorem \ref{thm1}}. We first deal with the
case $m=2$. If $n=2$ and $c=1$, then by (\ref{1.1}), we compute that
$$S(c2^{n+1},3)-S(c2^{n},3)=S(8,3)-S(4,3)=960,
$$
which implies that $\nu_2(S(8,3)-S(4,3))=6$. If either $n\ge3$ or
$c\ge3$, then $c2^n>n+2+\nu_2(c)$ since $c$ is odd. So by
Theorem \ref{thm0} (i), one knows that
$$\nu_2(S(c2^{n+1},3)-S(c2^{n},3))=n+1+\nu_2(c)=n+1
$$
as desired. So Theorem \ref{thm1} is proved for the case $m=2$.

In what follows we assume that $m\ge 3$. We proceed with induction
on $m$. Consider the case $m=3$. Since $n\ge3$ and $c$ is odd, one has
$c2^n>n+3+\nu_2(c)$. So we can apply Theorem \ref{thm0} (ii) and get that
$$
\nu_2(S(c2^{n+1},7)-S(c2^{n},7))=n+1+\nu_2(c)=n+1,
$$
which implies that Theorem \ref{thm1} is true for $m=3$.

Now let $m\ge 4$. Assume that Theorem \ref{thm1}
is true for the $m-1$ case. Then
$\nu_2(S(c2^{n+1}, 2^{m-1}-1)-S(c2^{n}, 2^{m-1}-1))=n+1$.
In the following we prove that Theorem \ref{thm1}
is true for the $m$ case. Write
\begin{align}\label{03} \Delta:= {2^m-1\choose
2^{m-1}}\big(S(c2^{n+1},2^m-1)-S(c2^n, 2^m-1)\big).
\end{align}
Then Lemma \ref{lem01} gives us that
\begin{align}\label{04}
\nu_2\big({2^{m}-1\choose
2^{m-1}}\big)=s_2(2^{m-1}-1)+s_2(2^{m-1})-s_2(2^{m}-1) =0.
\end{align}
It follows from (\ref{03}) and (\ref{04}) that
\begin{align*}
\nu_2(\Delta)=\nu_2(S(c2^{n+1},2^m-1)-S(c2^n, 2^m-1)).
\end{align*}
So to prove that Theorem \ref{thm1} is true for the
$m$ case, it is sufficient to show that
\begin{align} \label{q16}
\Delta\equiv 2^{n+1}\mod{2^{n+2}},
\end{align}
which will be done in what follows.

By Lemma \ref{lem11},  we obtain that
\begin{align}\label{q17}
{2^m-1\choose 2^{m-1}}S(c2^{n+1},2^m-1)&=\sum_{i=2^{m-1}-1}^{c2^{n+1}-2^{m-1}}
{c2^{n+1}\choose i}S(i,2^{m-1}-1)S(c2^{n+1}-i, 2^{m-1})\nonumber\\
&=\sum_{i=2^{m-1}-1,\atop \nu_2(i)\le m-2}^{c2^{n+1}-2^{m-1}}f_{m-1,n+1}(i)+
\sum_{i=2^{m-1}-1, \atop \nu_2(i)\ge m-1}^{c{2^{n+1}}-2^{m-1}}f_{m-1,n+1}(i)
\end{align}
and
\begin{align}
{2^m-1\choose
2^{m-1}}S(c2^{n},2^m-1)&=\sum_{i=2^{m-1}-1}^{c2^{n}-2^{m-1}}
{c2^{n}\choose i}S(i,2^{m-1}-1)S(c2^{n}-i, 2^{m-1})\nonumber\\
&=\sum_{i=2^{m-1}-1,\atop \nu_2(i)\le
m-3}^{c2^{n}-2^{m-1}}f_{m-1,n}(i) +\sum_{i=2^{m-1}-1, \atop
\nu_2(i)\ge m-2}^{c{2^{n}-2^{m-1}}}f_{m-1,n}(i),
\end{align}
where $f_{m-1,n}(i)$ and $f_{m-1,n+1}(i)$ are defined as in
(\ref{ad3}). Then letting $r=m-1$ and $s=n+1$ in Lemma \ref{lem13}
(iii) gives us that
\begin{align}
\nu_2\Big(\sum_{i=2^{m-1}-1,\atop \nu_2(i)\le
m-2}^{c2^{n+1}-2^{m-1}} f_{m-1,n+1}(i)\Big)\ge n+2.
\end{align}
Using Lemma \ref{lem13} (i) with $r=m-1$ and $s=n$, we deduce that
\begin{align}\label{q18}
\nu_2\Big(\sum_{i=2^{m-1}-1,\atop \nu_2(i)
\le m-3}^{c2^{n}-2^{m-1}}f_{m-1,n}(i)\Big)
\ge \min_{2^{m-1}-1\le i\le c2^{n}-2^{m-1}\atop \nu_2(i)
\le m-3}\{f_{m-1,n}(i)\}\ge n+2.
\end{align}
Then by (\ref{03}) and  (\ref{q17})-(\ref{q18}), we conclude that
\begin{align}\label{q19}
\Delta&\equiv\sum_{i=2^{m-1}-1, \atop \nu_2(i)\ge
m-1}^{c{2^{n+1}}-2^{m-1}}f_{m-1,n+1}(i)-\sum_{i=2^{m-1}-1, \atop
\nu_2(i)\ge m-2}^{c{2^{n}-2^{m-1}}}f_{m-1,n}(i)\mod{2^{n+2}}.
\end{align}

On the other hand, we have
\begin{align}
\sum_{i=2^{m-1}-1, \atop \nu_2(i)\ge m-1}^{c{2^{n+1}}-2^{m-1}}f_{m-1,n+1}(i)
&=\sum_{i=2^{m-1}-1, \atop \nu_2(i)=m-1}^{c{2^{n+1}}-2^{m-1}}f_{m-1,n+1}(i)
+\sum_{i=2^{m-1}-1, \atop \nu_2(i)>m-1}^{c{2^{n+1}}-2^{m-1}}
f_{m-1,n+1}(i)\nonumber\\
&=\sum_{l=0}^{c2^{n-m+1}-1}f_{m-1,n+1}(l2^m+2^{m-1})+
\sum_{l=1}^{c2^{n-m+1}-1}f_{m-1,n+1}(l2^m)
\end{align}
and
\begin{align}
\sum_{i=2^{m-1}-1, \atop \nu_2(i)\ge m-2}^{c{2^{n}-2^{m-1}}}f_{m-1,n}(i)
&=\sum_{i=2^{m-1}-1, \atop \nu_2(i)=m-2}^{c{2^{n}}-2^{m-1}}f_{m-1,n}(i)
+\sum_{i=2^{m-1}-1, \atop \nu_2(i)>m-2}^{c{2^{n}}-2^{m-1}}
f_{m-1,n}(i)\nonumber\\
&=\sum_{l=1}^{c2^{n-m+1}-2}f_{m-1,n}(l2^{m-1}+2^{m-2})+
\sum_{l=1}^{c2^{n-m+1}-1}f_{m-1,n}(l2^{m-1}).
\end{align}
Using Lemma \ref{lem13} (iv) with $s=n+1$ and $r=m-1$, we obtain
that
\begin{align}
f_{m-1,n+1}(2^{m-1})+f_{m-1,n+1}(c2^{n+1}-2^{m-1})\equiv 2^{n+1}\mod {2^{n+2}}.
\end{align}
Letting $s=n$ and $r=m-1$ in Lemma \ref{lem13} (v), we derive that
\begin{align}\label{q20}
\sum_{l=1}^{c2^{n-m+1}-2}\big(f_{m-1,n+1}(l2^m+2^{m-1})
-f_{m-1,n}(l2^{m-1}+2^{m-2})\big)\equiv
0\mod {2^{n+2}}.
\end{align}
Let $L:=\{l\in\mathbb{N}|1\le l\le c2^{n-m+1}-1\}$.
We define the following three subsets of $L$:
$$L_1=\{l\in L|\nu_2(l)<n-m+1\},\,
L_2=\{l\in L|\nu_2(l)=n-m+1\},\, L_3=\{l\in L|\nu_2(l)>n-m+1\}.$$
Then $L_1, L_2$ and $L_3$ are disjoint and $L=L_1\cup L_2\cup L_3$.
For $i\in\{1,2,3\}$, we define
$$\Delta_i:=\sum_{l\in L_i}\big(f_{m-1,n+1}(l2^m)-f_{m-1,n}(l2^{m-1})\big).$$
It then follows from (\ref{q19}) to (\ref{q20}) that
\begin{align}\label{q21}
\Delta&\equiv2^{n+1}+\sum_{l\in L}\Big(f_{m-1,n+1}(l2^m)
-f_{m-1,n}(l2^{m-1})\Big)\nonumber
\\&\equiv 2^{n+1}+\Delta_1+\Delta_2+\Delta_3 \mod{2^{n+2}}.
\end{align}
We claim that
\begin{align}\label{q22}
\Delta_1+\Delta_2+\Delta_3\equiv0\mod{2^{n+2}}.
\end{align}
Then from (\ref{q21}) and the claim (\ref{q22}), we deduce that
$\Delta\equiv2^{n+1}\mod{2^{n+2}}$, i.e. (\ref{q16}) is true. In
other words, Theorem \ref{thm1} holds for the $m$ case. It remains
to show that (\ref{q22}) is true that we will do in the following.

From Lemma \ref{lem02}, we obtain that
\begin{align}\label{q23}
{c2^{n+1}\choose l2^m}\equiv{c2^{n}\choose l2^{m-1}}\mod{2^{n+1}},
\end{align}
and Theorem \ref{lem10} tells us that
\begin{align}\label{q24}
\nu_2(S(c2^{n+1}-l2^m,2^{m-1})-S(c2^{n}-l2^{m-1},2^{m-1}))=
\nu_2(c2^{n-m+1}-l)+2.
\end{align}
By the inductive hypothesis, we infer that
\begin{align}\label{q25}
\nu_2(S(l2^m,2^{m-1}-1)-S(l2^{m-1},2^{m-1}-1))=\nu_2(l)+m.
\end{align}
It then follows from (\ref{q23}) to (\ref{q25}) and Lemma \ref{lem07}
that
\begin{align}\label{q26}
f_{m-1,n+1}(l2^m)=&{c2^{n+1}\choose l2^m}S(l2^m,2^{m-1}-1)
S(c2^{n+1}-l2^m,2^{m-1})\nonumber\\
=&\Big({c2^{n}\choose l2^{m-1}}+b_1\cdot 2^{n+1}\Big)
\Big(S(l2^{m-1},2^{m-1}-1)+(2b_2+1)\cdot 2^{\nu_2(l)+m}\Big)\nonumber\\
&\cdot\Big(S(c2^{n}-l2^{m-1},2^{m-1})+(2b_3+1)\cdot 2^{\nu_2(c2^{n-m+1}-l)+2}\Big),
\nonumber\\
\equiv&{c2^{n}\choose l2^{m-1}}\Big(S(l2^{m-1},2^{m-1}-1)
+(2b_2+1)\cdot 2^{\nu_2(l)+m}\Big)\nonumber\\
&\cdot\Big(S(c2^{n}-l2^{m-1},2^{m-1})+(2b_3+1)
\cdot 2^{\nu_2(c2^{n-m+1}-l)+2}\Big) \mod {2^{n+2}},
\end{align}
where $b_1,b_2,b_3\in \mathbb{N}$.

We first treat with $\Delta_1$. For any elment $l\in L_1$, by Lemma
\ref{lem12} (i), we get that
\begin{align}\label{q27}
\nu_2\Big({c2^{n}\choose l2^{m-1}}\Big)\ge n-m+1-\nu_2(l).
\end{align}
From Lemma \ref{lem07}, we derive that
\begin{align}\label{q28}
\nu_2\big(S(l2^{m-1},2^{m-1}-1)\cdot2^{\nu_2(l)+2}\big)=
\nu_2\big(S(c2^{n}-l2^{m-1},2^{m-1})\cdot 2^{\nu_2(l)+m}\big)=\nu_2(l)+m.
\end{align}
So by (\ref{q28}), we have
\begin{align}\label{q29}
\nu_2\big(S(l2^{m-1},2^{m-1}-1)\cdot2^{\nu_2(l)+2}+S(c2^{n}-l2^{m-1},
2^{m-1})\cdot 2^{\nu_2(l)+m}\big)\ge\nu_2(l)+m+1.
\end{align}
It then follows from (\ref{q26}), (\ref{q27}), (\ref{q29}), Lemma \ref{lem07}
and the fact $\nu_2(c2^{n-m+1}-l)=\nu _2(l)$ that
\begin{align*}
f_{m-1,n+1}(l2^m)\equiv&{c2^{n}\choose l2^{m-1}}
\Big(S(l2^{m-1},2^{m-1}-1)+2^{\nu_2(l)+m}\Big)
\Big(S(c2^{n}-l2^{m-1},2^{m-1})+2^{\nu_2(l)+2}\Big)\\
\equiv&{c2^{n}\choose l2^{m-1}}\Big(S(l2^{m-1},2^{m-1}-1)
\cdot2^{\nu_2(l)+2}+S(c2^{n}-l2^{m-1},2^{m-1})\cdot2^{\nu_2(l)+m}\Big)
\\
&+f_{m-1,n}(l2^{m-1})\equiv f_{m-1,n}(l2^{m-1})\mod{2^{n+2}}.
\end{align*}
That is, for each $l\in L_1$, one has
\begin{align}\label{q30}
f_{m-1,n+1}(l2^m)-f_{m-1,n}(l2^{m-1})\equiv 0\mod{2^{n+2}}.
\end{align}
Thus by (\ref{q30}), we deduce that
\begin{align}\label{q3001}
\Delta_1\equiv 0\mod{2^{n+2}}.
\end{align}

Consequently, we consider $\Delta_2$. Take an $l\in L_2$. One may let
$l=l_12^{n-m+1}$ with $1\le l_1\le c-1$ being odd. By Lemma
\ref{lem01}, we obtain that
\begin{align}\label{q31}
\nu_2\Big({c2^{n}\choose l2^{m-1}}\Big)
&=\nu_2\Big({c2^{n}\choose l_12^n}\Big)\nonumber\\
&=s_2(l_12^{n})+s_2((c-l_1)2^n)-s_2(c2^n)\nonumber\\
&=s_2(l_1)+s_2(c-l_1)-s_2(c).
\end{align}
It follows from Lemma \ref{lem07} and (\ref{q31}) that
\begin{align}\label{q32}
\nu_2\Big({c2^{n}\choose l2^{m-1}} 2^{\nu_2(l)+m}S(c2^{n}-l2^{m-1},
2^{m-1})\Big)=n+1+s_2(l_1)+s_2(c-l_1)-s_2(c).
\end{align}

Since $c$ is odd, one has $\nu_2(c2^{n-m+1}-l)\ge n-m+2$. Then
from (\ref{q26}), (\ref{q31}) and Lemma \ref{lem07}, we deduce that
\begin{align}\label{q33}
f_{m-1,n+1}(l2^m)&\equiv{c2^{n}\choose l2^{m-1}}\Big(S(l2^{m-1},
2^{m-1}-1)+ 2^{\nu_2(l)+m}\Big)S(c2^{n}-l2^{m-1},2^{m-1})\nonumber\\
&\equiv f_{m-1,n}(l2^{m-1})+{c2^{n}\choose l2^{m-1}} 2^{\nu_2(l)+m}
S(c2^{n}-l2^{m-1},2^{m-1})\mod {2^{n+2}}.
\end{align}
Note that $l_1=\frac{l}{2^{n-m+1}}$. Thus by (\ref{q32}) and
(\ref{q33}), we deduce that for each $l\in L_2$, we have
\begin{align}\label{q34}
f_{m-1,n+1}(l2^m)-f_{m-1,n}(l2^{m-1})\equiv
\left\{
\begin{array}{lc}
2^{n+1} & {\rm if} \ s_2(\frac{l}{2^{n-m+1}})
+s_2(c-\frac{l}{2^{n-m+1}})=s_2(c),\\
0 & {\rm otherwise}
\end{array}
\right.
\mod{2^{n+2}}.
\end{align}
It follows from (\ref{q34}) that
\begin{align}\label{q3401}
\Delta_2\equiv \sum_{l\in
L'_2}(f_{m-1,n+1}\big(l2^m)-f_{m-1,n}(l2^{m-1})\big)\equiv
2^{n+1}|L'_2|\mod{2^{n+2}},
\end{align}
where
$$L'_2:=\{l\in L_2|s_2(\frac{l}{2^{n-m+1}})
+s_2(c-\frac{l}{2^{n-m+1}})=s_2(c)\}.$$

Now we deal with $\Delta_3$. For any $l\in L_3,$ since $c$
is odd, one has $\nu_2(c2^{n-m+1}-l)=n-m+1$ and $1\le
c2^{n-m+1}-l\le c2^{n-m+1}-1$. Then there exists an odd integer
$1\le l_2\le c-1$ such that $c2^{n-m+1}-l=l_22^{n-m+1}$.
By Lemma \ref{lem01}, we get that
\begin{align}\label{q35}
\nu_2\Big({c2^{n}\choose l2^{m-1}}\Big)&=\nu_2\Big({c2^{n}\choose
(c2^{n-m+1}-l)2^{m-1}}\Big)\nonumber\\
&=\nu_2\Big({c2^{n}\choose l_22^n}\Big)\nonumber\\
&=s_2(l_2)+s_2(c-l_2)-s_2(c).
\end{align}
Furthermore, by Lemma \ref{lem07} and (\ref{q35}), we obtain that
\begin{align}\label{q36}
\nu_2\Big({c2^{n}\choose l2^{m-1}} 2^{n-m+3}S(l2^{m-1},
2^{m-1}-1)\Big)=n+1+s_2(l_2)+s_2(c-l_2)-s_2(c).
\end{align}

On the other hand, since $\nu_2(l)>n-m+1$ and $\nu_2(c2^{n-m+1}-l)=n-m+1$,
then by (\ref{q26}) and Lemma \ref{lem07} we know that
\begin{align}\label{q37} f_{m-1,n+1}(l2^m)&\equiv{c2^{n}\choose
l2^{m-1}}S(l2^{m-1},2^{m-1}-1)
\Big(S(c2^{n}-l2^{m-1},2^{m-1})+2^{\nu_2(c2^{n-m+1}-l)+2}\Big)\nonumber\\
&\equiv f_{m-1,n}(l2^{m-1})+{c2^{n}\choose l2^{m-1}} 2^{n-m+3}
S(l2^{m-1},2^{m-1}-1)\mod {2^{n+2}}.
\end{align}
Since $l_2=c-\frac{l}{2^{n-m+1}}$, by (\ref{q36}) and (\ref{q37}),
one deduces that for each $l\in L_3$,
\begin{align}\label{q38}
f_{m-1,n+1}(l2^m)-f_{m-1,n}(l2^{m-1})\equiv \left\{
\begin{array}{lc}
2^{n+1}& {\rm if} \ s_2(\frac{l}{2^{n-m+1}})+s_2(c-\frac{l}{2^{n-m+1}})
=s_2(c),\\
0&{\rm otherwise}
\end{array}
\right.
\mod{2^{n+2}}.
\end{align}
It then follows from (\ref{q38}) that
\begin{align}\label{q3801}
\Delta_3\equiv \sum_{l\in
L'_3}(f_{m-1,n+1}\big(l2^m)-f_{m-1,n}(l2^{m-1})\big)\equiv
2^{n+1}|L'_3|\mod{2^{n+2}},
\end{align}
where
$$L'_3:=\{l\in L_3|s_2(\frac{l}{2^{n-m+1}})+s_2(c-\frac{l}{2^{n-m+1}})=s_2(c)\}.$$

Let $\Psi: L'_2\rightarrow  L'_3$ be a mapping defined for any $l\in L'_2$
by $\Psi(l):=c2^{n-m+1}-l.$ Obviously, $\Psi$ is well defined and injective.
Take an element $\tilde{l}\in L'_3$. Then $1\le \tilde{l}\le c2^{n-m+1}-1$,
$\nu_2(l)>n-m+1$ and
$$s_2(\frac{\tilde{l}}{2^{n-m+1}})+s_2(c-\frac{\tilde{l}}{2^{n-m+1}})=s_2(c).$$
Let $l=c2^{n-m+1}-\tilde{l}.$ It follows that $1\le l\le c2^{n-m+1}-1$,
$\nu_2(l)=n-m+1$ and $$s_2(\frac{l}{2^{n-m+1}})
+s_2(c-\frac{l}{2^{n-m+1}})=s_2(c-\frac{\tilde{l}}{2^{n-m+1}})
+s_2(\frac{\tilde{l}}{2^{n-m+1}})=s_2(c).$$
Thus $l\in L'_2$. So $\Psi$ is surjective. Hence $\Psi$ is
bijective, which means that $|L'_2|=|L'_3|$. It then follows from
(\ref{q22}), (\ref{q3001}), (\ref{q3401}) and (\ref{q3801}) that
\begin{align*}
\Delta_1+\Delta_2+\Delta_3&\equiv 2^{n+1}(|L'_2|+|L'_3|)\equiv 0 \mod{2^{n+2}}
\end{align*}
as claimed. Hence (\ref{q22}) is proved.

This completes the proof of Theorem
\ref{thm1}. \hfill$\Box$

\end{document}